\documentclass{patmorin}
\usepackage[T1]{fontenc}
\usepackage[utf8]{inputenc}
\usepackage{amsmath}
\usepackage{amsfonts}
\usepackage{amsthm}
\usepackage{graphicx}
\usepackage{enumerate}
\usepackage{pat}
\usepackage{hyperref}
\usepackage{paralist}
\title{\MakeUppercase{Dual Circumference and Collinear Sets}%
    \thanks{This work was partly funded by NSERC and MRI.}}

\author{Vida Dujmovi\'c%
        \thanks{School of Computer Science and Electrical Engineering,
                University of Ottawa}\,\, and
        Pat Morin%
        \thanks{School of Computer Science, Carleton University}}

\newcommand{\dual}[1]{{#1}^\star}

\begin{document}
\maketitle

\begin{abstract}
   We show that, if an $n$-vertex triangulation $G$ of maximum degree $\Delta$ has a dual that contains a cycle of length $\ell$, then $G$ has a non-crossing straight-line drawing in which some set, called a \emph{collinear set}, of $\Omega(\ell/\Delta^4)$ vertices lie on a line.  Using the current lower bounds on the length of longest cycles in cubic 3-connected graphs, this implies that every $n$-vertex planar graph of maximum degree $\Delta$ has a collinear set of size $\Omega(n^{0.8}/\Delta^4)$.
\end{abstract}

\section{Introduction}

Throughout this paper, all graphs are simple and finite and have at least
4 vertices.  For a planar graph $G$, we say that a set $S\subseteq V(G)$
is a \emph{collinear set} if $G$ has a non-crossing straight-line drawing
in which the vertices of $S$ are all collinear.  A \emph{plane graph} is
a planar graph $G$ along with a particular non-crossing drawing of $G$.
The \emph{dual} $\dual{G}$ of a plane graph $G$ is the graph whose
vertex set $V(\dual{G})$ is the set of faces in $G$ and in which $fg\in
E(\dual{G})$ if and only if the faces $f$ and $g$ of $G$ have at least
one edge in common.  The \emph{circumference}, $c(G)$, of a graph $G$
is the length of the longest cycle in $G$. In \secref{proof}, we prove
the following theorem:

\begin{thm}\thmlabel{main}
  Let $G$ be a triangulation of maximum degree $\Delta$ whose dual
  $\dual{G}$ has circumference $\ell$. Then $G$ has a collinear set of
  size $\Omega(\ell/\Delta^4)$.
\end{thm}

The dual of a triangulation is a 3-connected cubic planar graph.
The study of the circumference of 3-connected cubic planar graphs
has a long and rich history going back to at least 1884 when Tait
\cite{tait:remarks} conjectured that every such graph is Hamiltonian.  In
1946, Tait's conjecture was disproved by Tutte who gave a non-Hamiltonian
46-vertex example \cite{tutte:on}.  Repeatedly replacing vertices of
Tutte's graph with copies of itself gives a family of graphs, $\langle G_i:i\in
\Z\rangle$ in which $G_i$ has $46\cdot 45^i$ vertices and circumference at
most $45\cdot44^i$.  Stated another way, $n$-vertex members of the
family have circumference $O(n^\alpha)$, for $\alpha=\log_{44}(45) < 0.9941$.
The current best upper bound of this type is due to Gr\"unbaum and
Walther \cite{grunbaum.walther:shortness} who construct a 24-vertex
non-Hamiltonian cubic 3-connected planar graph, resulting in a family
of graphs in which $n$-vertex members have circumference $O(n^{\alpha})$
for $\alpha=\log_{23}(22)< 0.9859$.

A series of results has steadily improved the lower bounds on the
circumference of $n$-vertex  (not necessarily planar) 3-connected
cubic graphs.  Barnette \cite{barnette:trees} showed that, for
every $n$-vertex 3-connected cubic graph $G$, $c(G)=\Omega(\log n)$.
Bondy and Simonovits \cite{bondy.simonovits:longest} improved this bound
to $e^{\Omega(\sqrt{\log n})}$ and conjectured that it can be improved
to $\Omega(n^\alpha)$ for some $\alpha>0$.  Jackson \cite{jackson:longest}
confirmed this conjecture with $\alpha=\log_2(1+\sqrt{5})-1 > 0.6942$.
Billinksi \etal\ \cite{bilinksi.jackson.ea:circumference} improved
this to the solution of $4^{1/\alpha}-3^{1/\alpha}=2$, which implies
$\alpha>0.7532$.  The current record is held by Liu, Yu, and Zhang
\cite{liu.yu.zhang:circumference} who show that $\alpha>0.8$.

It is known that any planar graph of maximum degree $\Delta$ can be
triangulated so that the resulting triangulation has maximum degree
$\lceil 3\Delta/2\rceil+11$ \cite{kant.bodlaender:triangulating}. This
fact, together with \thmref{main} and the result of Liu, Yu, and Zhang
\cite{liu.yu.zhang:circumference}, implies the following corollary:

\begin{cor}\corlabel{main}
  Every $n$-vertex planar graph of maximum degree $\Delta$ contains a
  collinear set of size $\Omega(n^{0.8}/\Delta^4)$.
\end{cor}


It is known that every planar graph $G$ has a collinear set of size
$\Omega(\sqrt{n})$ \cite{bose.dujmovic.ea:polynomial,dujmovic:utility}
. \corref{main} therefore improves on this bound for bounded-degree planar
graphs and, indeed for the family of $n$-vertex planar graphs of maximum
degree $\Delta\in O(n^{\delta})$, with $\delta < 0.075$.  For example,
the triangulations dual to Gr\"unbaum and Walther's construction have
maximum degree $\Delta \in O(\log n)$.  As discussed below, this implies
that there exists $n$-vertex triangulations of maximum degree $O(\log
n)$ whose largest collinear set has size $O(n^{0.9859})$.  \corref{main}
implies that every $n$-vertex planar graph of maximum degree $O(\log n)$
has a collinear set of size $\Omega(n^{0.8})$.

Recently, Dujmovi\'c \etal\ \cite{dujmovic.frati.ea:every} have shown
that every collinear set is \emph{free}. That is, for any planar graph
$G$, any collinear set $S\subseteq V(G)$, and any set $X\subset\R^2$
with $|X|=|S|$, there exists a non-crossing straight-line drawing of $G$
in which the vertices of $S$ are drawn on the points of $X$.  Because of
this, collinear sets have immediate applications in graph drawing and
related areas.  For applications of \corref{main}, including untangling
\cite{cibulka:untangling,pach.tardos:untangling,watanabe:open,goaoc.kratochvil.ea:untangling,kang.pikhurko.ea:untangling,bose.dujmovic.ea:polynomial,dalozzo.dujmovic.ea:drawing,dujmovic:utility,ravsky.verbitsky:on},
column
planarity~\cite{barba.evans.ea:column,evans.kusters.ea:column,dalozzo.dujmovic.ea:drawing,dujmovic:utility},
universal point
subsets~\cite{digiacomo.liotta.ea:how,angelini.binucci.ea:universal,dalozzo.dujmovic.ea:drawing,dujmovic:utility},
and partial simultaneous geometric
drawings~\cite{evans.kusters.ea:column,barba.hoffmann.ea:column,angelini.evans.ea:sefe,blasius.kobourov.ea:simultaneous,dujmovic:utility}
the reader is referred to Dujmovi\'c \cite{dujmovic:utility}
and Dujmovi\'c \etal\ \cite[Section~1.1]{dujmovic.frati.ea:every}.
\corref{main} gives improved bounds for all of these problems for planar
graphs of maximum $\Delta\in o(n^{0.075})$.

For example, it is known that every $n$-vertex planar geometric graph
can be untangled while keeping some set of $\Omega(n^{0.25})$ vertices
fixed \cite{bose.dujmovic.ea:polynomial} and that there are $n$-vertex
planar geometric graphs that cannot be untangled while keeping any set of
$\Omega(n^{0.4948})$ vertices fixed \cite{cano.toth.ea:upper}. Although
asymptotically tight bounds are known for paths \cite{cibulka:untangling},
trees \cite{goaoc.kratochvil.ea:untangling}, outerplanar graphs
\cite{goaoc.kratochvil.ea:untangling}, planar graphs of treewidth
two \cite{ravsky.verbitsky:on}, and planar graphs of treewidth three
\cite{dalozzo.dujmovic.ea:drawing}, progress on the general case has
been stuck for 10 years due to the fact that the exponent $0.25$ comes
from two applications of Dilworth's Theorem.  Thus, some substantially
new idea appears to be needed. By relating collinear/free sets to
dual circumference, the current paper presents an effective new idea.
Indeed, \corref{main} implies that every bounded-degree $n$-vertex
planar geometric graph can be untangled while keeping $\Omega(n^{0.4})$
vertices fixed.  Even for bounded-degree planar graphs,
$\Omega(n^{0.25})$ was the best previously-known lower bound.

Our work opens two avenues for further progress:

\begin{enumerate}
   \item Lower bounds on the circumference of
      3-connected cubic graphs are an active area of research. At the time of writing, the $\Omega(n^{0.8})$ lower bound of Liu, Yu, and Zhang
      \cite{liu.yu.zhang:circumference} is less than a year old.  Any
      further progress on these lower bounds will translate immediately
      to an improved bound in \corref{main} and all its applications.

   \item It is possible that the dependence on $\Delta$ can be removed
      from \thmref{main} and \corref{main}, thus making these results
      applicable to all planar graphs, regardless of maximum degree.
\end{enumerate}

\section{Proof of \thmref{main}}
\seclabel{proof}

Let $G$ be a plane graph.  We treat the vertices of $G$ as points,
the edges of $G$ as closed curves, and the faces of $G$ as closed sets
(so that a face contains all the edges on its boundary and an edge
contains both its endpoints).  Whenever we consider subgraphs of $G$
we treat them as having the same embedding as $G$.  Similarly, if we
consider a graph $G'$ that is homeomorphic\footnote{We say that a graph $G'$ is homeomorphic to $G$ if $G'$ can be obtained from $G$ by repeatedly contracting an edge of $G$ that is incident to a degree-2 vertex.} to $G$ then we assume
that the edges of $G'$---each of which represents a path in $G$
whose internal vertices all have degree 2---inherit their embedding from
the paths they represent in $G$.

Finally, if we consider the dual $\dual{G}$ of $G$ then we treat it
as a plane graph in which each vertex $f$ is represented as a point in
the interior of the face $f$ of $G$ that it represents.  The edges of
$\dual{G}$ are embedded so that an edge $fg$ is contained in the union of
the two faces $f$ and $g$ of $G$, it intersects the interior of exactly
one edge of $G$ that is common to $f$ and $g$, and this intersection
consists of a single point.

A \emph{proper good curve} $C$ for a plane graph $G$ is a
Jordan curve with the following properties:
\begin{enumerate}
	\item[\emph{proper}:] for any edge $xy$ of $G$, $C$ either contains $xy$, intersects
  $xy$ in a single point (possibly an endpoint), or is disjoint
  from $xy$; and
  \item[\emph{good}:] $C$ contains at least one point in the interior of
  some face of $G$.
\end{enumerate}

Da Lozzo \etal\ \cite{dalozzo.dujmovic.ea:drawing} show that proper good
curves define collinear sets:

\begin{thm}\thmlabel{dalozzo}
  In a plane graph $G$, a set $S\subseteq V(G)$ is a collinear set if
  and only if there is a proper good curve for $G$ that contains $S$.
\end{thm}

For a triangulation $G$, let $v(G)$ denote the size of a largest
collinear set in $G$.  We will show that, for any triangulation $G$
of maximum degree $\Delta$
whose dual is $\dual{G}$, $v(G)=\Theta(c(\dual{G})/\Delta^4)$ by relating proper good curves in $G$ to cycles in $\dual{G}$.

As shown by Ravsky and Verbitsky
\cite{ravsky.verbitsky:on,ravsky.verbitsky:on-arxiv}, the inequality $v(G)
\le c(\dual{G})$ is easy: If $G$ is a triangulation that has a proper
good curve $C$ containing $k$ vertices, then a slight deformation of
$C$ produces a proper good curve that contains no vertices. This curve
intersects a cyclic sequence of faces $f_0,\ldots,f_{k'-1}$ of $G$
with $k'\ge k$.  In this sequence, $f_i$ and $f_{(i+1)\bmod k'}$ share
an edge, for every $i\in\{0,\ldots,k'-1\}$, so this sequence is a closed
walk in the dual $\dual{G}$ of $G$.  The properness of the original curve
and the fact that each face of $G$ is a triangle ensures that $f_i\neq
f_j$ for any $i\neq j$, so this sequence is a cycle in $\dual{G}$ of
length $k'\ge k$.  Therefore, $c(\dual{G})\ge v(G)$. From the result
of Gr\"unbaum and Walther described above, this implies that there are
$n$-vertex triangulations $G$ such that $v(G) = O(n^{0.9859})$.

The other direction, lower-bounding $v(G)$ in terms $c(\dual{G})$
is more difficult. Not every cycle $C$ of length $\ell$ in $\dual{G}$
can be easily transformed into a proper good curve containing a similar
number of vertices in $C$.  In the next section, we describe three
parameters $\tau$, $\rho$, and $\kappa$ of a cycle $C$ in $\dual{G}$
and show that $C$ can always be transformed into a proper good curve
containing $\Omega(\kappa)$ vertices of $G$.

\subsection{Faces that are Touched, Pinched, and Caressed}

Throughout the remainder of this paper, $G$ is a triangulation
whose dual is $\dual{G}$ and $C$ is a cycle in $\dual{G}$.
Refer to \figref{touched_pinched_caressed} for the following definitions.
We say that a face $f$ of $\dual{G}$
\begin{enumerate}
  \item is \emph{touched} by $C$ if $f\cap C\neq \emptyset$;
  \item is \emph{pinched} by $C$ if $f\cap C$ is a cycle or has more than
    one connected component; and
  \item is \emph{caressed} by $C$ if it is touched but not pinched by $C$.
\end{enumerate}

\begin{figure}
  \begin{center}
    \includegraphics{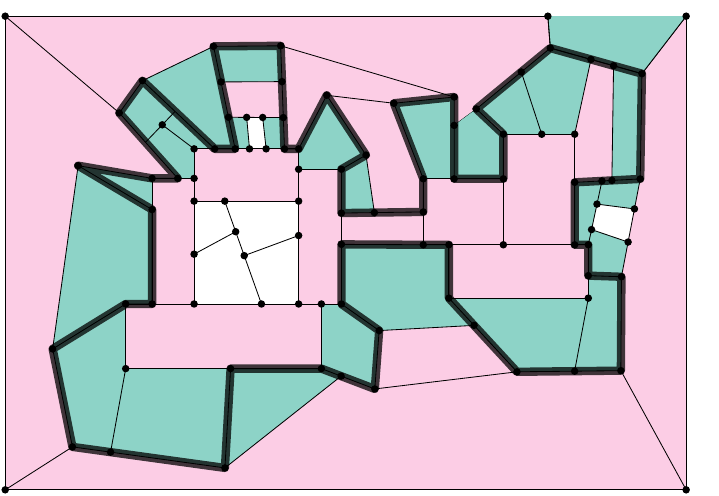}
  \end{center}
  \caption{Faces of $\dual{G}$ that are pinched and caressed by $C$. $C$
  is bold, caressed faces are teal, pinched faces are pink, and untouched
  faces are unshaded.} \figlabel{touched_pinched_caressed}
\end{figure}

Since $C$ is almost always the cycle of interest, we will usually say
that a face $f$ of $\dual{G}$ is touched, pinched, or caressed, without
specifically mentioning $C$.  We will frequently use the values $\tau$,
$\rho$, and $\kappa$ to denote the number of faces of $\dual{G}$ in
some region that are $\tau$ouched, $\rho$inched or $\kappa$aressed.
Observe that, since every face that is touched is either pinched or
caressed, we have the identity $\tau = \rho + \kappa$.

\begin{lem}\lemlabel{cycle_to_curve}
   If $C$ caresses $\kappa$ faces of $\dual{G}$ then $G$ has a proper
   good curve that contains at least $\kappa/4$ vertices so, by
   \thmref{dalozzo}, $v(G)\ge \kappa/4$.
\end{lem}

\begin{proof}
  Let $F$ be the set of faces in $\dual{G}$ that are caressed by $C$. Each
  element $u\in F$ corresponds to a vertex of $G$ so we will treat $F$ as
  a set of vertices in $G$.  Consider the subgraph $G[F]$ of $G$ induced
  by $F$.  The graph $G[F]$ is planar and has $\kappa$ vertices. Therefore,
  by the 4-Colour Theorem \cite{robertson.seymour.ea:four-colour}, $G[F]$
  contains an independent set $F'\subseteq F$ of size at least $\kappa/4$.

  We claim that there is a proper good curve for $G$ that contains all
  the vertices in $F'$.  To see this, first observe that the cycle $C$ in
  $\dual{G}$ already defines a proper good curve (that does not contain
  any vertices of $G$) that we also call $C$.  We perform
  local modifications on $C$ so that it contains all the vertices in $F'$.

  For any vertex $u\in F'$, let $w_0,\ldots,w_{d-1}$ denote the neighbours
  of $u$ in cyclic order.  The curve $C$ intersects some contiguous
  subsequence $uw_i,\ldots,uw_j$ of the edges adjacent to $u$.  Since $u$
  is caressed, this sequence does not contain all edges incident
  to $u$. Therefore, the curve $C$ crosses the edge $w_{i-1}w_i$, then
  crosses $uw_i,\ldots,uw_j$, and then crosses the edge $w_j w_{j+1}$.
  We modify $C$ by removing the portion between the first and last of
  these crossings and replacing it with a curve that contains $u$ and is
  contained in the two faces $w_{i-1}uw_i$ and $w_juw_{j+1}$. (See
  \figref{cycle_to_curve}.)

  \begin{figure}
     \begin{center}
	\begin{tabular}{cc}
		\includegraphics{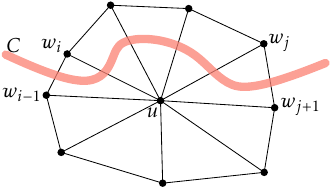} &
		\includegraphics{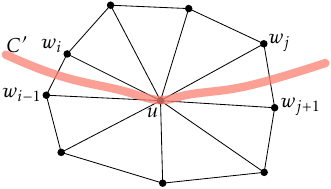}
	\end{tabular}
     \end{center}
     \caption{Transforming the dual cycle $C$ into a proper good curve $C'$ containing $u$.}
	  \figlabel{cycle_to_curve}
  \end{figure}

  After performing this local modification for each $u\in F'$ we have
  a curve $C'$ that contains every vertex $u\in F'$.  All that remains
  is to verify that $C'$ is good and proper for $G$. That $C'$ is good for
  $G$ is obvious.  That $C'$ is proper for $G$ follows from the following two
  observations: (i)~$C'$ does not contain any two adjacent vertices (since
  $F'$ is an independent set); and (ii)~if $C'$ contains a vertex $u$,
  then it does not intersect the interior of any edge incident to $u$.
\end{proof}

\lemref{cycle_to_curve} reduces our problem to finding a cycle in
$\dual{G}$ that caresses many faces.  It is tempting to hope that
any sufficiently long cycle in $\dual{G}$ caresses many faces, but
this is not true; \figref{few_caressed} shows that even a Hamiltonian
cycle $C$ in $\dual{G}$ may caress only four faces, two inside $C$ and
two outside of $C$.  In this example, there is an obvious sequence of
faces $f_0,\ldots,f_k$, all contained in the interior of $C$ where $f_i$
shares an edge with $f_{i+1}$ for each $i\in\{0,\ldots,k-1\}$.  The only
caressed faces in the interior of $C$ are the endpoints $f_0$ and $f_k$ of this sequence.

\begin{figure}
   \begin{center}
       \includegraphics{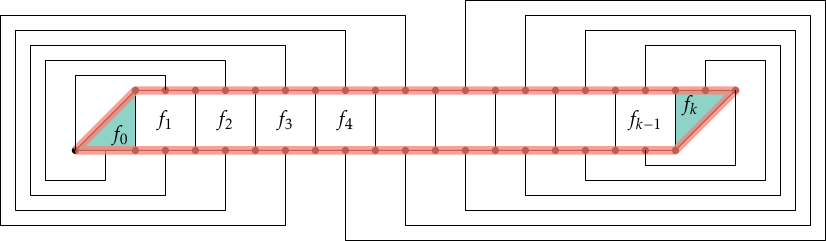}
   \end{center}
   \caption{A Hamiltonian cycle $C$ in $\dual{G}$ that caresses only four faces.}
   \figlabel{few_caressed}
\end{figure}

Our strategy is to define a tree structure, $T_0$ on groups of faces
contained in the interior of $C$ and a similar structure, $T_1$ on groups
of faces in the exterior of $C$.  We will then show that every leaf of
$T_0$ or $T_1$ contains a face caressed by $C$. In \figref{few_caressed},
the tree $T_0$ is the path $f_0,\ldots,f_k$ and, indeed, the leaves $f_0$
and $f_k$ of this tree are caressed by $C$.  After a non-trivial analysis of the trees $T_0$ and $T_1$, we will eventually show that,
if $C$ does not caress many faces, then $T_0$ and $T_1$ have many nodes,
but few leaves.  Therefore $T_0$ and $T_1$ have many degree-2 nodes.
This abundance of degree-2 nodes makes it possible to perform a \emph{surgery}
on $C$ that increases the number of caressed faces.  Performing this
surgery repeatedly will then produce a curve $C$ that caresses many faces.

A path $P=v_1,\ldots,v_r$ in $\dual{G}$ is a \emph{chord
path} (for $C$) if $v_1,v_r\in V(C)$ and $v_2,\ldots,v_{r-1}\not\in
V(C)$.  Note that this definition implies that the interior vertices
$v_2,\ldots,v_{r-1}$ of $P$ are either all contained in the interior of
$C$ or all contained in the exterior of $C$.

\begin{lem}\lemlabel{one_caressed}
   Let $P$ be a chord path for $C$ and let $L$ and $R$ be the two faces
   of the graph formed by $P\cup C$ that each contain $P$ in their boundary. Then $R$
   contains at least one face of $\dual{G}$ that is caressed by $C$.
\end{lem}

\begin{proof}
   The proof is by
   induction on the number, $t$, of faces of $\dual{G}$ contained in $R$.
   If $t=1$, then $R$ is a face of $\dual{G}$ and it is caressed by $C$.

   If $t>1$, then consider the face $f$ of $\dual{G}$ that is contained
   in $R$ and has the first edge of $P$ on its boundary.  Refer to
   \figref{one_caressed}. Since $t>1$, $X=R\setminus f$ is non-empty. The
   set $X$ may have several connected components $X_1,\ldots,X_k$, but
   each $X_i$ has a boundary that contains a chord path $P_i$ for $C$.
   We can therefore apply induction on $P_1$ (or any $P_i$) using $R=X_1$
   in the inductive hypothesis.
  \begin{figure}
     \begin{center}
	\begin{tabular}{cc}
		\includegraphics{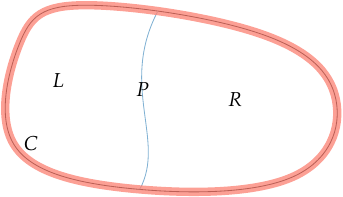} &
		\includegraphics{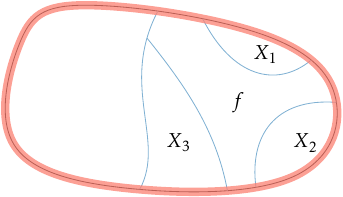}
	\end{tabular}
     \end{center}
	  \caption{The proof of \lemref{one_caressed}.}
	  \figlabel{one_caressed}
  \end{figure}
\end{proof}

\subsection{Auxilliary Graphs and Trees: $H$, $\tilde{H}$, $T_0$, and $T_1$}

Refer to \figref{auxilliary}. Consider the auxilliary
graph $H$ with vertex set $V(H)\subseteq V(\dual{G})$ and whose edge set
consist of the edges of $C$ plus those edges of $\dual{G}$ that belong
to any face pinched by $C$. Let $v_0,\ldots,v_{r-1}$ be the clockwise cyclic sequence of vertices on some face $f$ of $\dual{G}$ that is pinched by $C$.
We identify three kinds of vertices that are \emph{special} with respect to $f$:
(see \figref{keepers}).
\begin{enumerate}
  \item A vertex $v_i$ is special of \emph{Type~A} if $v_{i-1}v_i$ is an edge of $C$ and $v_iv_{i+1}$ is not an edge of $C$.
  \item A vertex $v_i$ is special of \emph{Type~B} if $v_{i-1}v_i$ is not an edge of $C$ and $v_iv_{i+1}$ is an edge of $C$.
  \item A vertex $v_i$ is special of \emph{Type~Y} if $v_i$ not incident to any edge of $C$ and $v_i$ has degree 3 in $H$.
\end{enumerate}

  \begin{figure}
     \begin{center}\begin{tabular}{cc}
		\includegraphics[width=.45\textwidth]{figs/t0t1-2} &
		\includegraphics[width=.45\textwidth]{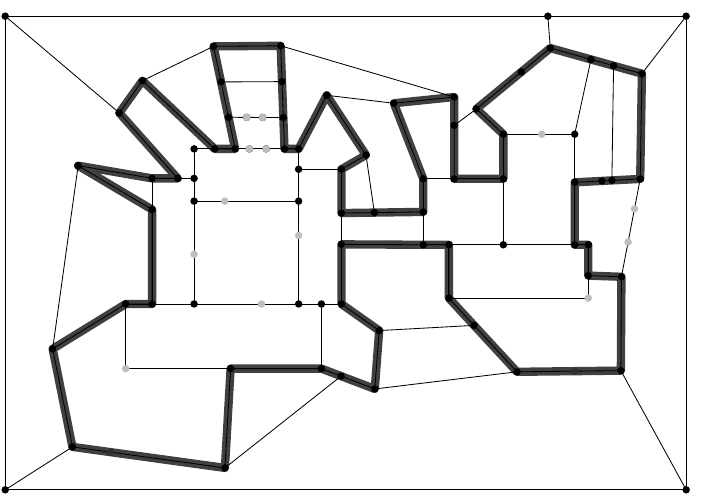} \\
                (a) & (b) \\[1em]
		\includegraphics[width=.45\textwidth]{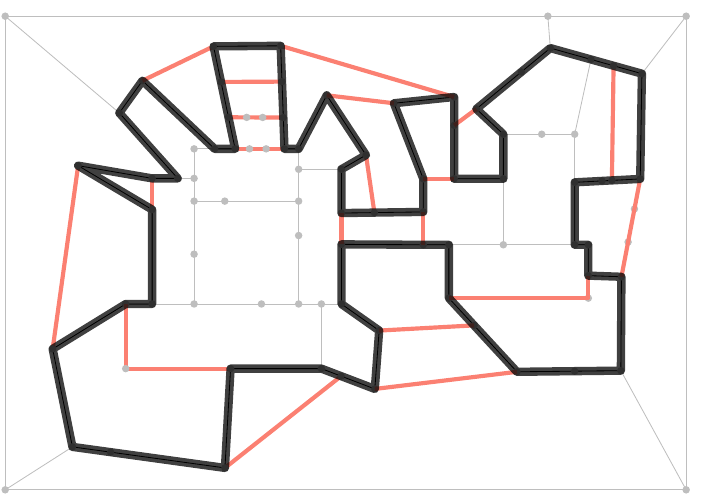} &
		\includegraphics[width=.45\textwidth]{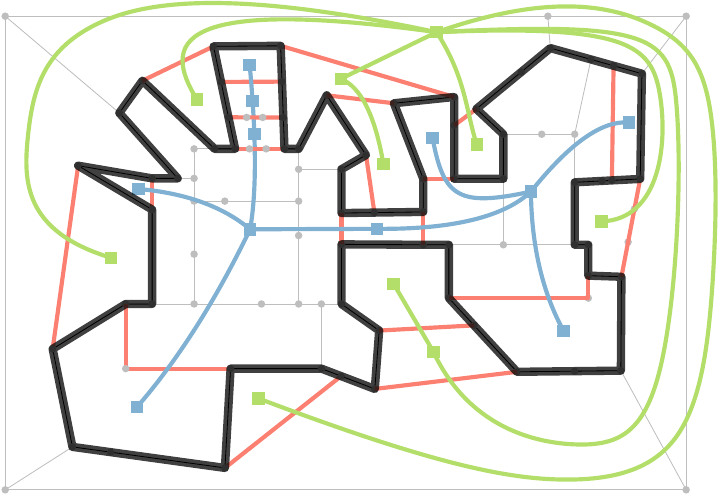} \\
                (c) & (d)
     \end{tabular}\end{center}
	  \caption{(a)~the cycle $C$ in $\dual{G}$ with faces classified as pinched or caressed; (b)~the auxilliary graph $H$; (c)~the auxilliary graph $\tilde{H}$ with keeper paths highlighted; (d)~the trees $T_0$ and $T_1$.}
	  \figlabel{auxilliary}
  \end{figure}

We say that a chord path $v_i,\ldots,v_j$ is a \emph{keeper} with respect to $f$ if $v_i$ is special of Type~A, $v_j$ is special of Type~B, and none of $v_{i+1},\ldots,v_{j-1}$ are special.  We let $\tilde{H}$ denote the subgraph of $H$ containing all the edges of $C$ and all the edges of all paths that are keepers with respect to some pinched face $f$ of $\dual{G}$.

\begin{figure}
   \begin{center}\begin{tabular}{ccc}
        \includegraphics{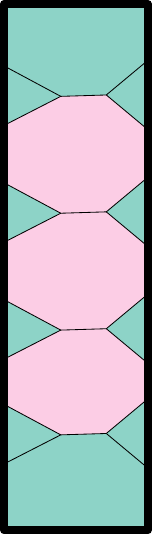} &
        \includegraphics{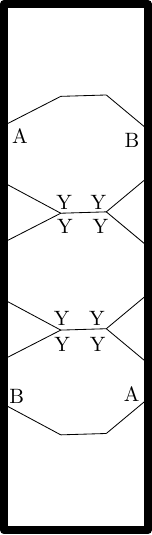} &
        \includegraphics{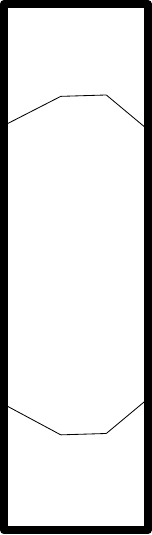} \\
        $\dual{G}$ & $H$ & $\tilde{H}$
   \end{tabular}\end{center}
    \caption{The graphs $\dual{G}$, $H$, and $\hat{H}$ and the classification of special vertices of types $A$, $B$, and $Y$.}
    \figlabel{keepers}
\end{figure}

It is worth emphasizing at this point that, by definition, every keeper is
entirely contained in the boundary of at least one face $f$ of $\dual{G}$.
This property will be useful shortly.

Let $\tilde{H}'$ denote the graph that is homeormophic to $\tilde{H}$
but does not contain any degree 2 vertices.  That is, $\tilde{H}'$
is the minor of $\tilde{H}$ obtained by repeatedly contracting an edge
incident a degree-2 vertex.  The graph $\tilde{H}'$ naturally inherits an
embedding from the embedding of $\tilde{H}$.  This embedding partitions
the edges of $\tilde{H}'$ into three sets:
\begin{enumerate}
  \item The set $B$ of edges that are contained in (the embedding of) $C$;
  \item The set $E_0$ of edges whose interiors are contained in the
  interior of (the embedding of) $C$; and
  \item The set $E_1$ of edges whose interiors are contained in the
  exterior of (the embedding of) $C$.
\end{enumerate}

Observe that, for each $i\in\{0,1\}$, the graph $H_i$ whose edges are
exactly those in $B\cup E_i$ is outerplanar, since all vertices of $H_i$
are on a single face, whose boundary is $C$.  Let $\dual{H_i}$ be dual
of $H_i$ and let $T_i$ be the subgraph of $\dual{H_i}$ whose edges are
all those dual to the edges of $E_i$. From the outerplanarity of $H_i$,
it follows that $T_i$ is a tree.

Each vertex of $T_i$ corresponds to a face of $\tilde{H}$.  From this
point onwards, we will refer to the vertices of $T_i$ as \emph{nodes}
to highlight this fact, so that a node $u$ of $T_i$ is synonymous with
the subset of $\R^2$ contained in the corresponding face of $\tilde{H}$.
In the following, when we say that a node $u$ of $T_i$ contains a face $f$
of $\dual{G}$ we mean that $f$ is one of the faces of $\dual{G}$ whose
union makes up $u$.  The degree, $\delta_u$ of any node $u$ in $T_i$
is exactly equal to the number of keeper paths on the boundary of $u$.

The following lemma allows us to direct our effort towards proving that
one of $T_0$ or $T_1$ has many leaves.

\begin{lem}\lemlabel{one_caressed_leaf}
   Each leaf $u$ of $T_i$ contains at least one face of $\dual{G}$
   that is caressed by $C$.
\end{lem}

\begin{proof}
   The edge of $T_i$ incident to $u$ corresponds to a chord path $P$. The
   graph $P\cup C$ has two faces with $P$ on its boundary, one of which
   is $u$.  The lemma now follows immediately from \lemref{one_caressed},
   with $R=u$.
\end{proof}

We will make use of the following well-known property of
3-connected plane graphs.

\begin{lem}\lemlabel{one_shared_edge}
   If $G$ has $n\ge 4$ vertices then any two faces of $\dual{G}$ share at
   most one edge.
\end{lem}

\begin{proof}
   Suppose that two faces $f$ and $g$ share two edges $e_1$ and
   $e_2$. Then $e_1$ and $e_2$ form an edge cutset of $\dual{G}$.
   If $\dual{G}$ contains at least four vertices, then two of the
   endpoints of $e_1$ and $e_2$ form a vertex cutset of $\dual{G}$
   of size 2, contradicting the fact that $\dual{G}$ is 3-connected.
   That $\dual{G}$ contains at least four vertices follows from Euler's
   Formula, which gives the number of vertices in $\dual{G}$ as $2n-4\ge
   4$ for all $n\ge 4$.
\end{proof}

Note that, as should be evident from \figref{keepers}, the number of faces in $\tilde{H}$ is not lower bounded by any function of the number of faces in $H$ and therefore the number of nodes in $T_0$ and $T_1$ is not lower bounded by any function of $\ell$.  Indeed, a single face of $\tilde{H}$ may contain arbitrarily many faces of $\dual{G}$ that are touched by $C$.  The following important lemma shows that, when this happens, the corresponding node in $T_0$ or $T_1$ either has high degree or contains many faces of $\dual{G}$ that are caressed by $C$.  The latter case is obviously good for our purposes. The former case is also good because a vertex of degree $\delta$ in any tree creates $\delta-2$ leaves and, by \lemref{one_caressed_leaf}, each leaf contains at least one caressed face.

For a node $u$ of $T_i$, we let $\tau_u$, $\rho_u$, $\kappa_u$, and $\delta_u$ denote the number of touched face of $\dual{T}$ in $u$, pinched faces of $\dual{G}$ in $u$, the number of caressed faces of $\dual{G}$ in $u$, and the degree of $u$ in $T_i$, respectively.

\begin{lem}\lemlabel{many_caressed_or_high_degree}
   For any node $u$ of $T_i$, $\rho_u \le 2(\kappa_u+\delta_u)$.
\end{lem}

Before proving \lemref{many_caressed_or_high_degree}, we point out that
the leading constant 2 is tight. \figref{tight} shows an example in
which all $\rho_u=2k+1$ pinched faces of $\dual{G}$ are contained in a
single (pink) node $u$ of $T_0$ that contains $\kappa_u=0$ caressed faces and
has degree $\delta_u=k+2$.

\begin{figure}
  \begin{center}
    \includegraphics{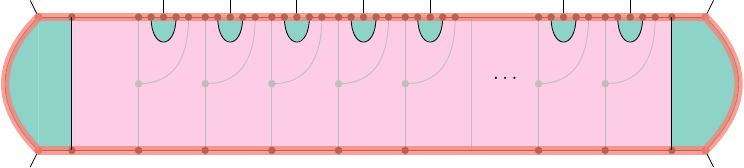}
  \end{center}
  \caption{An example showing the tightness of \lemref{many_caressed_or_high_degree}.}
  \figlabel{tight}
\end{figure}


\begin{proof}[Proof of \lemref{many_caressed_or_high_degree}]
   The proof is a discharging argument.  We assign each pinched face in
   $u$ a single unit of charge, so that the total charge is $\rho_u$.
   We then describe a discharging procedure that preserves the total
   charge and such that, after executing this procedure, the folowing conditions are satisfied:
   \begin{compactenum}[(Post1)]
        \item Each pinched face has no charge.
        \item Each caressed face has charge at most 2.
        \item Each keeper path has charge at most 2.
    \end{compactenum}
    Since there is a bijection between keeper paths in $u$ and edges of $T_i$ incident to $u$, this proves the result.

    The discharging procedure is made up of two routines, an \emph{initialization procedure} and a \emph{recursive procedure}.  The recursive procedure takes inputs $(L,R,P,c)$, where $P$ is a chord path, $L,R\subseteq u$, $L\cap R=P$, $L$ contains at least one face of $\dual{G}$, and $0\le c\le 2$ is a charge that we think of as resting on $P$.  The input $(L,R,P,c)$ must satisfy the following conditions:
   \begin{compactenum}[(Pre1)]
        \item Each face of $\dual{G}$ in $L$ that shares an edge with $P$
        is pinched.
        \item If $c>1$ then $P$ is contained in the boundary of a single face of $\dual{G}$ that is contained in $L$.
    \end{compactenum}
    The procedure guarantees that, after its completion, the charge of $c$ that was resting on $P$ has been moved into $R$, any other charges in $L$ are undisturbed, and the faces contained in $R$ satisfy (Post1)--(Post3).

    Before defining the recursive procedure itself, we will show how it is used by the initialization procedure.  This initialization procedure takes an arbitrary pinched face $f$ contained in $u$.  Since $f$ is pinched, it has $r\ge 2$ chord paths $P_1,\ldots,P_r$ on its boundary. For each $i\in\{1,\ldots,r\}$, let $L_i^-$ be the component of $u\setminus P_i$ that contains $f$, let $L_i=L_i^-\cup P_i$, and let $R_i=u\setminus L_{i^-}$.
    This initialization procedure guarantees that, after it runs, all the faces and chord paths in $u\setminus R_1$ satisfy (Post1)--(Post3) but does not modify charges on faces and keeper paths in $R_1$.

    The initialization procedure works as follows: Since $f$ is pinched it has a charge of 1 so we move the charge from $f$ onto $P_2$ and apply the recursive procedure to $(L_2,R_2,P_2,1)$.  Since $f$ is pinched, this satisifies (Pre1) and since the final argument $c=1$ this satisfies (Pre2).  Once these recursive procedures are complete conditions (Post1)--(Post2) are satisfied for all faces in $f\cup R_2$.

    Next, we apply the recursive procedure on $(L_i,R_i,P_i,0)$ for each $i\in\{3,\ldots,r\}$.  Since $f$ is a pinched face, this satisifies (Pre1) and since the final argument $c=0$ this satisfies (Pre2).  Once the recursive procedure is complete conditions (Post1)--(Post2) are satisfied for all faces in $f\cup R_i$ and does affect any charges in $R_1$.

    Since every face and keeper path contained in $u$ is contained in $R_i$ for at most one $i$, the initialization procedure produces a distribution of charges that satisifies (Post1)--(Post3) for $u\setminus R_1$, as required.

    Next we describe the recursive discharging procedure that takes $(L,R,P,c)$ satisifying (Pre1) and (Pre2) and moves charges in $R$, and the charge $c$ resting on $P$, so that they satisfy (Post1)--(Post3).  There are several cases to consider (see \figref{discharging}):
	\begin{figure}
		\begin{center}
		\begin{tabular}{cc}
			\includegraphics{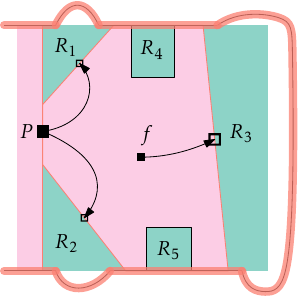} &
			\includegraphics{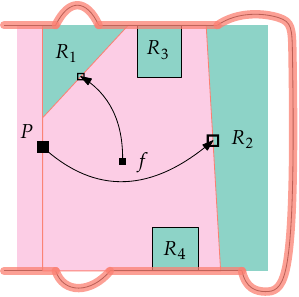} \\
			 2.a & 2.b \\[1.5em]
			\includegraphics{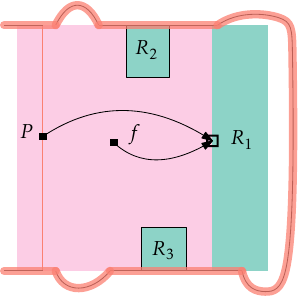} &
			\includegraphics{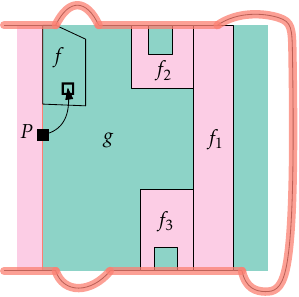} \\
			 2.c & 3
		\end{tabular}
		\end{center}
		\caption{Discharging steps in the proof of
		\lemref{many_caressed_or_high_degree}.}
		\figlabel{discharging}
	\end{figure}
  \begin{enumerate}
     \item $R$ contains no face of $\dual{G}$ that is pinched by $C$.
     If $R$ contains no face of $\dual{G}$ at all, then $R=P$ is a keeper path,
     in which case we leave a charge of $c$ on it and we are done.
     Otherwise $R$ contains at least one face of $\dual{G}$ and \lemref{one_caressed}  ensures that $R$ contains at least one caressed face $f$.  We move the charge from $P$ onto $f$ and we are done.

     \item $R$ contains a face $f$ of $\dual{G}$ that is pinched by $C$ and that shares
     at least one edge with $P$.  We consider three subcases, each illustrated in \figref{discharging}:
     \begin{enumerate}
	\item $f$ contains neither endpoint of $P$. In this case,
	$R\setminus f$ has two distinct components, $R_1^-$ and $R_2^-$ each containing a distinct endpoint of $P$.  For each $i\in\{1,2\}$, let $P_i$ be the chord path that separates $R_i^-$ from $u\setminus R_i^-$.  Since $f$ is pinched, $f$ contains $r\ge 3$ chord paths $P_1,\ldots,P_{r}$.  Indeed, if $P_1$ and $P_2$ were the only chord paths on $f$, then $f$ would be caressed.  For each $i\in\{1,\ldots,r\}$, let $L_i^-=u\setminus P_i$, let $L_i=L_i^-\cup P_i$, and let $R_i=u\setminus L_i^-$.

    We split the charge $c$ on $P$ evenly between $P_1$ and $P_2$ and apply the recursive procedure on $(L_i,R_i,P_i,c/2)$ for each $i\in\{1,2\}$. Next, we move the charge on $f$ to $P_3$ and apply the recursive procedure on $(L_3,R_3,P_3,1)$.  Finally, we apply the recursive procedure on $(L_i,R_i,P_i,0)$ for each $i\in\{4,\ldots,r\}$.

    The recursive call $(L_1,R_1,P_1,c/2)$ satisfies (Pre1) because the path $P_1$ used in this recursive call is contained in the boundary of $f$ and $P$. In particular each face of $\dual{G}$ contained in $u\setminus R_1$ that is incident to $P_1$ is either in $L$ and incident to $P$ or is the face $f$.  The latter faces are pinched by (Pre1) and $f$ is pinched by definition.  The recursive call on $(L_1,R_1,P_1,c/2)$ also satisifies (Pre2) since $c\le 2$, so $c/2\le 1$.  The same argument shows that the recursive call on $(L_2,R_2,P_2,c/2)$ satisfies (Pre1) and (Pre2).

    For each $i\in\{3,\ldots,r\}$, the recursive call on $(L_i,R_i,P_i,\star)$ satisifies (Pre1) because $P_i$ is contained in $f$ and $f$ is pinched and satisfies (Pre2) because the final argument is $1$ for $i=3$ and $0$ for $i\in\{4,\ldots,r\}$.

	\item $f$ contains exactly one endpoint of $P$.  In this case, $R\setminus f$ has one connected component $R_1^-$ that contains an endpoint of $P$.  Since $f$ is pinched, $f$ has $r\ge 2$ chord paths $P_1,\ldots,P_r$ on its boundary, where $P_1$ separates $R_1^-$ from $u\setminus R_1$. Define $L_1,\ldots,L_r$, $R_1,\ldots,R_r$, and $P_2,\ldots,P_r$ as in the previous case.

    Because $f$ is pinched, it has one unit of charge on it, that we move onto $P_1$ before calling the recursive procedure on $(L_1,R_1,P_1,1)$. This satisfies (Pre1) for the same reasons described in the previous case and satisfies (Pre2) because the final argument is $1$.

    The path $P$ has a charge $c\le 2$ which we move onto $P_2$ and call the recursive procedure on $(L_2,R_2,P_2,c)$.  This recursive call satisfies (Pre1) because $f$ is pinched and it satisfies (Pre2) because $P_2$ is entirely contained in the boundary of $f$.

    Finally, for each $i\in\{3,\ldots,r\}$, we call the recursive procedure on $(L_i,R_i,P_i,0)$. Clearly each of these calls also satisfies (Pre1) and (Pre2).

    \item $f$ contains both endpoints of $P$.  We claim that, in this case, $P$ must be on the boundary of more than one face in $L$, otherwise $P$ would be a keeper path. To see this, observe that the face $f$ contains both the first edge $e_1$ and last edge $e_2$ of $P$. If $e_1=e_2$ because $P$ is a single edge, then it is certainly a keeper, which is not possible since $P$ is in the interior of $u$. Otherwise, by \lemref{one_shared_edge}, $e_1$ and $e_2$ are on the boundary of two different faces in $L$.

    Therefore, by (Pre2) $P$ has $c\le 1$ units of charge assigned to it.
    Now, since $f$ is pinched, it has $r\ge 1$ chord paths $P_1,\ldots,P_r$, other than $P$ on its boundary.  Define $L_1,\ldots,L_r$ and $R_1,\ldots,R_r$ as in the previous two cases.  Now, $P$ has a charge $c\le 1$ and, since it is pinched, $f$ has a charge of 1.  We move these $c+1$ units of charge from $P$ and $f$ onto $P_1$ and call the recursive procedure on $(L_1,R_1,P_1,c+1)$.  This satisfies (Pre1) since $f$ is pinched and satisifies (Pre2) since $P_1$ is entirely contained in the boundary of $f$.

    For each $i\in\{2,\ldots,r\}$ we then call the recursive procedure on $(L_i,R_i,P_1,0)$. Clearly each of these calls satisfies (Pre1) and (Pre2).

  \end{enumerate}
  \item $R$ contains at least one pinched face of $\dual{G}$, but no pinched face in $R$ shares an edge with $P$.  We claim that there is a single face, $g$ of $H$, contained in $R$, that contains all of $P$ on its boundary.
  Indeed, edges of $\dual{G}$ not in $C$ are in $H$ only if they are on the boundary of some pinched face of $\dual{G}$.  Since no pinched face of $\dual{G}$ in $R$ shares an edge with $P$, none of the edges incident to internal vertices of $P$ and contained in $R$ are part of $H$.  Therefore, $P$ is on the boundary of a single face of $H$ that is contained in $R$.

  Let $f$ be the face of $\dual{G}$ that is contained in $R$ and that contains the first edge of $P$.  The face $f$ is touched by $C$ but not pinched, so it must be caressed.  We move the $c$ units of charge from $P$ onto $f$.

  Now, $R$ still contains one or more pinched faces  $f_1,\ldots,f_k$, such that each $f_i$ shares part of a chord path $P_i$ with $g$.  Consider one such $f_i$ and observe that $u\setminus f_i$ has $r_i\ge 2$ chord paths $P_{i,1},\ldots,P_{i,r_i}$ on its boundary and use the convention that $P_{i,1}=P_i$.  Define $L_{i,1},\ldots,L_{i,r_i}$ and $R_{i,1},\ldots,R_{i,r_i}$ in a manner analagous to $L_1,\ldots,L_r$ and $R_1,\ldots,R_r$ in the previous cases.

  On each such face $f_i$, we run the \emph{initialization} procedure and this reorganizes the charges in $u\setminus R_{i,1}$ so that they satisfy (Post1)--(Post3) and does not modify charges in $L\cup g$.  Doing this for each $i\in\{1,\ldots,k\}$ completes the description of the discharging procedure.
\end{enumerate}
  To complete the proof, first observe that if $u$ contains no pinched faces then the result is trivially true. Otherwise $u$ contains a pinched face $f$ such that one of the components $R_1$ of $u\setminus f$ contains no pinched faces.  (The existence of such an $f$ is established by choosing $f$ so that the minimum number of faces in any component of $u\setminus f$ is minimum over all pinched faces $f$ in $u$.)  Since $R_1$ contains no pinched faces, it contains no charges, so it already satisfies (Post1)--(Post3). Running the initialization procedure on $f$ will then redistribute charges so that they satisfy (Post1)--(Post3) for all faces and keeper paths in $u$.
\end{proof}

\subsection{Bad Nodes}


We say that a node of $T_i$ is \emph{bad} if it has degree 2 and
contains no face of $\dual{G}$ that is caressed by $C$.  We now move
from studying individual nodes of $T_0$ and $T_1$ to studying global
quantities associated with $T_0$ and $T_1$.  From this point on, for
each $i\in\{0,1\}$,
\begin{enumerate}
  \item $\tau_i$, $\rho_i$, and $\kappa_i$ refer the total numbers of faces contained in nodes of $T_i$ that are touched, pinched, and caressed by $C$, respectively;
  \item $n_i$ refers to the number of nodes of $T_i$;
  \item $\delta_i = 2(n_i-1)$ is the total degree of all nodes in $T_i$; and
  \item $b_i$ is the number of bad nodes in $T_i$.
\end{enumerate}

\begin{lem}\lemlabel{few_caressed_implies_many_nodes}
  If $\kappa_i \le \tau_i/6$ then $n_i\ge \tau_i/8$.
\end{lem}

\begin{proof}
	From \lemref{many_caressed_or_high_degree} we know $\rho_i \le 2(\kappa_i+\delta_i)$, so
  \[
  \tau_i  = \kappa_i + \rho_i
     \le 3\kappa_i + 2\delta_i
     = 3\kappa_i + 4(n_i-1)
     \le \tau_i/2 + 4n_i \enspace ,
  \]
  and reorganizing the left- and right-hand sides gives the desired result.
\end{proof}

\begin{lem}\lemlabel{most_nodes_are_bad}
   For any $0<\epsilon < 1$, if $b_i \le (1-\epsilon)n_i$, then
   $\kappa_i \ge \epsilon\tau_i/24$.
\end{lem}

\begin{proof}
   Partition the nodes of $T_i$ into the following sets:
   \begin{enumerate}
       \item the set $B$ of bad nodes;
       \item the set $N_1$ of leaves;
       \item the set $N_{\ge 3}$ of nodes having degree at least 3;
       \item the set $N_2$ of nodes having degree 2 that are not bad.
   \end{enumerate}
   Then
   \begin{align*}
     b_i & = n_i - |N_1|- |N_{\ge 3}|-|N_2| \\
	   & > n_i - 2|N_1|-|N_2| & \text{since $|N_1|> |N_{\ge 3}|$} \\
         & \ge  n_i - 2\kappa_i -|N_2|
	   & \text{(since, by \lemref{one_caressed_leaf}, $\kappa_i\ge |N_1|$)} \\
           & \ge  n_i - 3\kappa_i
           & \text{(since each node in $N_2$ contains a caressed face)}
    \end{align*}
    Thus, we have
    \[
          n_i-3\kappa_i \le b_i \le (1-\epsilon)n_i
    \]
    and rewriting gives
    \begin{equation}
      \kappa_i \ge \epsilon n_i/3 \enspace . \eqlabel{blech}
    \end{equation}
    If $\kappa_i \ge \tau_i/6$, then the proof is complete since $\tau_i/6>\tau_i/24$.  On the other hand, if $\kappa_i \le \tau_i/6$ then, by \lemref{few_caressed_implies_many_nodes}, $n_i \ge \tau_i/8$. Combining this with \eqref{blech}  gives
    \[
      \kappa_i \ge \epsilon n_i/3 \ge \epsilon\tau_i/24  \enspace . \qedhere
    \]
\end{proof}

\subsection{Interactions Between Bad Nodes}

We have now reached a point in which we know that the vast
majority of nodes in $T_0$ and $T_1$ are bad nodes, otherwise
\lemref{most_nodes_are_bad} implies that a constant fraction of the
faces touched by $C$ are caressed by $C$.  At this point, we are ready
to study interactions between bad nodes of $T_0$ and bad nodes of $T_1$.

\begin{lem}\lemlabel{common_face}
   If $u$ is a bad node then $u$ is a face of $\dual{G}$.
\end{lem}

\begin{proof}
   First observe that, since $u$ is bad, it has degree 2, so $C\cap u$ has exactly two connected components $C_1$ and $C_2$. Thus $u$'s boundary
   consists of $C_1$, $C_2$ and two chord paths $P_1$ and $P_2$.  We first argue that there is a single face $g$ of $\dual{G}$ that contains $C_1\cup C_2$.  If not, then $\dual{G}$ must contain a path $P$ whose interior is in $u$ and has both endpoints on the boundary of $u$.  There are a few cases to consider:
   \begin{figure}
       \begin{center}
           \centering{
           \begin{tabular}{cc}
               \includegraphics{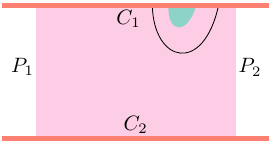} &
               \includegraphics{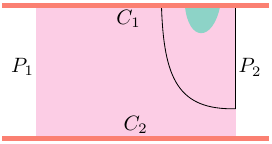} \\
               Case 1 & Case 2 \\[1ex]
               \includegraphics{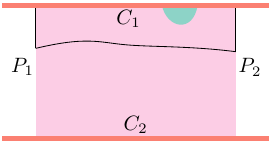} &
               \includegraphics{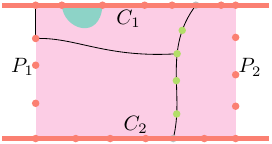} \\
               Case 3 & Case 4 \\[1ex]
               \multicolumn{2}{c}{\includegraphics{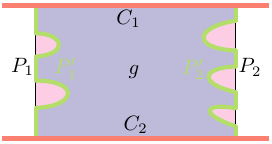}}
          \end{tabular}
           }
           \caption{Cases in the proof of \lemref{common_face}}
           \figlabel{common_face}
       \end{center}
   \end{figure}
   \begin{enumerate}
       \item $P$ has both endpoints on $C_i$ for some $i\in\{1,2\}$. In
       this case, $P$ is a chord path and, by \lemref{one_caressed}
       $u$ contains a face that is caressed by $C$, contradicting the
       assumption that $u$ is a bad node.

       \item $P$ has one endpoint on $C_i$ and one endpoint on $P_j$
       for some $i,j\in\{1,2\}$.  In this case, $P\cup P_j$ contains a
       chord path with both endpoints on $C_i$, again contradicting the
       assumption that $u$ is a bad node.

       \item $P$ has one endpoint on $P_1$ and one endpoint on $P_2$.
       In this case, $P\cup P_1\cup P_2$ contains a chord path with both
       endpoints on $C_1$, again contradicting the assumption that $u$
       is a bad node.

       \item $P$ has one endpoint on $C_1$ and one endpoint on $C_2$.
       The path $P$ is not a keeper, otherwise it would have split $u$
       into two nodes.	Therefore, it must be the case that $P$ contains
       an internal vertex.  Let $S_1$ be the set of internal vertices
       of $P$ and let $S_2$ be the set of vertices on the boundary
       of $u$, not including the endpoints of $P$.  Since $\dual{G}$
       is 3-connected, there is a path from $S_1$ to $S_2$ that does
       not contain either endpoint of $P$.  The shortest such path, $P'$,
       does not contain any edges of $P$.  Again, using portions of $P$,
       $P_1$, $P_2$, and $P'$ we can construct a chord path, contained
       in $u$, with both endpoints on $C_1$ or both endpoints on $C_2$,
       contradicting the assumption that $u$ is a bad node.
\end{enumerate}
    This establishes that $C_1\cup C_2$ is contained in the boundary of a single face $g$ of $\dual{G}$.  The boundary of $g$ contains two disjoint paths $P_1'$ and $P_2'$ joining $C_1$ and $C_2$.  We claim that $P_i'$ is a keeper path, for each $i\in\{1,2\}$.  Indeed, each internal $x$ vertex of $P_i'$ is either a vertex of $P_j$ or is on the boundary of three faces: $g$ and two faces that are not touched by $C$.  In either case, $x$ is not special of Type~Y.  Therefore $P_i'$ has endpoints that are special of Type~A and Type~B with respect to the pinched face $g$ and has no internal vertices that are special of Type~Y, so $P_i'$ is a keeper.  Therefore $\{P_1',P_2'\}=\{P_1,P_2\}$ since, otherwise $f$ would not be a face of $\tilde{H}$.  Therefore $g=f$ so $f$ is a face of $\dual{G}$.
\end{proof}

       %
       %

The following lemma shows that a bad node $u$ in $T_0$ and a bad node $w$
in $T_1$ share at most one edge of $C$.

\begin{lem}\lemlabel{bad_one_shared_edge}
   Any two bad nodes $u$ of $T_i$ and $w$ of $T_j$ have at most one edge in common.
\end{lem}

\begin{proof}
    By \lemref{common_face} $u$ and $w$ are each faces of $\dual{G}$. Therefore, by \lemref{one_shared_edge}, $u$ and $w$ share at most one edge.
\end{proof}

%
%

\subsection{Really Bad Nodes}

At this point we will start making use of the assumption that the
triangulation $G$ has maximum degree $\Delta$, which is equivalent to
the assumption that each face of $\dual{G}$ has at most $\Delta$ edges
on its boundary.

\begin{obs}\obslabel{degree_touched}
  If $G$ has maximum degree $\Delta$ and $C$ has length $\ell$, then
  the number of faces $\tau$ of $\dual{G}$ touched by $C$ is at least
  $2\ell/\Delta$.  At least $\ell/\Delta$ of these faces are in
  the interior of $C$ and at least $\ell/\Delta$ of these faces are
  in the exterior of $C$.
\end{obs}

\begin{proof}
  Orient the edges of $C$ counterclockwise so that, for each edge
  $e$ of $C$, the face of $\dual{G}$ to the left of $e$ is in $C$'s
  interior and the face of $\dual{G}$ to the right of $e$ is in $C$'s
  exterior.  Each face of $\dual{G}$ has at most $\Delta$ edges.
  Therefore, the number of faces to the right of edges in $C$ is
  at least $\ell/\Delta$. The same is true for the number of faces
  of $\dual{G}$ to the left of edges in $C$.
\end{proof}

For each node $u$ of $T_i$, we define $N(u)$ as the set of nodes in $T_0$
and $T_1$ (excluding $u$) that share an edge of $\dual{G}$ with $u$.
Note that $N(u)$ contains the neighbours of $u$ in $T_i$ as well as
nodes of $T_{1-i}$ with which $u$ shares an edge of $C$.

%

We say that a node $u$ is \emph{really bad} if $u$ and all nodes in $N(u)$ are bad.

\begin{lem}\lemlabel{lots_of_really_bad}
  For each $i\in\{0,1\}$ and each $0<\alpha < 1/24$,
  if $G$ has maximum degree $\Delta$, $C$ has length $\ell$, and the number
  $\kappa$, of faces of $\dual{G}$ caressed by $C$ is at most $\alpha\ell/\Delta$, then the number $b_i$ of really bad nodes in $T_i$ is at least $n_i - \alpha(120\Delta+72) n_i$.
\end{lem}

\begin{proof}
  Without loss of generality, let $i=0$. From \obsref{degree_touched}, we know that $\tau_0\ge \ell/\Delta$.  Therefore, $\kappa_0\le\kappa\le \alpha\ell/\Delta \le \alpha\tau_0 \le \tau_0/6$ so, by \lemref{few_caressed_implies_many_nodes}, $n_0 \ge
  \tau_0/8$.

  By \lemref{most_nodes_are_bad}, if $b_0 < (1-24\alpha)n_0$, then
  \[
      \kappa > \kappa_0 \ge \alpha\tau_0
      \ge \alpha\ell/\Delta
      \enspace .
  \]
  This violates our assumption that $\kappa \le \alpha\ell/\Delta$.
  Therefore, we may assume that $b_0\ge (1-24\alpha)n_0$.

  We now want to study how many of the bad nodes in $T_0$ are really bad. Let $A$ be the set of nodes in $T_0$ that are not bad and partition $A$ into $A_1$ (leaves), $A_2$ (degree-2 nodes) and $A_{\ge 3}$ (nodes of degree at least 3).  We make use of the
  following inequality:
	\begin{equation}
		|A_1| = 2 + \sum_{w\in A_{\ge 3}}(\delta_w-2) \ge \sum_{w\in A_{\ge 3}}(\delta_w-2)
		\ge \sum_{w\in A_{\ge 3}}\delta_w/3 \enspace ,  \eqlabel{triple}
	\end{equation}
  which is true because $x-2\ge x/3$ for all $x\ge 3$.

  Now each node $w$ in $A$ can prevent at most $\delta_w$ bad nodes of $T_0$ from being really bad.  We count this as follows:
  \[
   \sum_{w\in A}\delta_w
     = \sum_{w\in A_1}\delta_w
         + \sum_{w\in A_2}\delta_w
         + \sum_{w\in A_{\ge 3}}\delta_w \\
	   \le |A_1| + 2|A_2| + 3|A_1| \enspace .
  \]
  Now, $A_1$ contains leaves of $T_0$ and, by \lemref{one_caressed_leaf}, each leaf of $T_0$ contains a caressed face.  Therefore $|A_1|\le\kappa$.
  Next, $A_2$ contains degree-2 nodes of $T_0$ that are not bad.  If a node has degree-$2$ and contains no caressed face, then it is bad. Therefore each node in $A_2$ contains a caressed face. Therefore $|A_2|\le \kappa$, so $|A_1|+2|A_2|+3|A_1| \le 6\kappa$.  Picking up where we left off:
  \[
      \sum_{w\in A}\delta_w \le |A_1|+2|A_2|+3|A_1|
       \le 6\kappa
     \le 6\alpha\ell/\Delta \le 48\alpha n_0 \enspace ,
  \]
  where the last inequality uses the fact that $n_0 \ge
  \tau_0/8 \ge \ell/(8\Delta)$.
  That is, the set $A$ of non-bad nodes in $T_0$ prevents at most $48\alpha n_0$ bad nodes in $T_0$ from being really bad.  Next we account
  for nodes in $T_1$ that prevent bad nodes in $T_0$ from being really bad.

  Let $A'$ be the set of nodes in $T_1$ that are not bad.  For two
  nodes $u$ in $T_0$ and $w$ in $T_1$,  $w\in N(u)$ if and only if $w$
  and $u$ share an edge of $C$.  The number of edges of $C$ incident to a node $w$ is at most $\Delta\tau_w$.  Therefore, we can upper bound the number of bad nodes in $T_0$ that are prevented from being really bad by some node in $T_1$ as
  \begin{align*}
   \sum_{w\in A'}\Delta \tau_w \\
    & = \sum_{w\in A'}\Delta (\rho_w+\kappa_w)
    & \text{(since $\tau_w=\rho_w+\kappa_w$)} \\
    & \le  \sum_{w\in A'}(3\Delta\kappa_w + 2\Delta\delta_w) & \text{(by \lemref{many_caressed_or_high_degree})} \\
    & \le  3\Delta\kappa + \sum_{w\in A'}2\Delta\delta_w \\
    & < 3\Delta\kappa + 12\Delta\kappa & \text{(by defining $A'_1$, $A'_2$, $A'_{\ge 3}$ and arguing as above)}\\
    & = 15\Delta\kappa \\
    & \le 15\alpha\ell & \text{(since $\kappa\le\alpha\ell/\Delta$, by assumption)} \\
    & \le 120\alpha\Delta n_0 & \text{(since $n_0\ge \tau_0/8\ge \ell/(8\Delta)$)}
  \end{align*}
  Therefore, the number of bad nodes in $T_0$ is $b_0$ and the number of these that are really bad is at least
  \[
     b_0 - \alpha(120\Delta + 48)n_0 \ge
     n_0 - \alpha(120\Delta + 72)n_0 \enspace . \qedhere
  \]
\end{proof}

We say that a node $u$ is \emph{really really bad} if all the nodes in $N(u)$ are really bad.  (Note that this implies that $u$ is bad.)  The following lemma extends \lemref{lots_of_really_bad} to really really bad nodes:

\begin{lem}\lemlabel{lots_of_really_really_bad}
  For each $i\in\{0,1\}$ and each $0<\alpha < 1/24$,
  if $G$ has maximum degree $\Delta$, $C$ has length $\ell$, and the number
  $\kappa$, of faces of $\dual{G}$ caressed by $C$ is at most $\alpha\ell/\Delta$, then the number $b_i$ of really really bad nodes in $T_i$ is at least $n_i - \alpha(\Delta+1)(120\Delta+72) n_i= n_i - O(\alpha\Delta^2)$.
\end{lem}

\begin{proof}
    A node $u$ is a \emph{fringe node} if it is really bad but not really really bad. A node $u$ is a \emph{critical node} if it is bad but not really bad.  Observe that every fringe node $u$ is in $N(w)$ for some critical node $w$.  To bound the number of fringe nodes, it therefore suffices to bound $\sum_{w}|N(w)|\le\sum_{w}\Delta$ where the sum is over all critical nodes and the inequality is due to \lemref{common_face}, so $|N(w)|\le\Delta$ for any bad node $w$.

    By \lemref{lots_of_really_bad}, the number of nodes that are not really bad, and hence the number of critical nodes, is at most $\alpha(120\Delta+72) n_i$.  Therefore, the number of fringe nodes is at most $\alpha\Delta(120\Delta+72) n_i$.
    Any node that is not really really bad is either a fringe node or is not really bad.  Therefore, the number of nodes that are really really bad is at least
    \[
        n_i - \alpha(\Delta+1)(120\Delta+72) n_i \enspace . \qedhere
    \]
\end{proof}

The following observation, illustrated in \figref{three-paths}, follows from the fact that all the nodes it considers are bad and that $\tilde{H}$ is a cubic graph, so each vertex of $\tilde{H}$ is on the boundary of 3 faces. The second part of the figure shows an example in which $\{a_1,\ldots,a_r\}$ and $\{b_1,\ldots,b_s\}$ are not disjoint. (In this example $b_1=a_3$.)

\begin{figure}
    \centering{
        \includegraphics{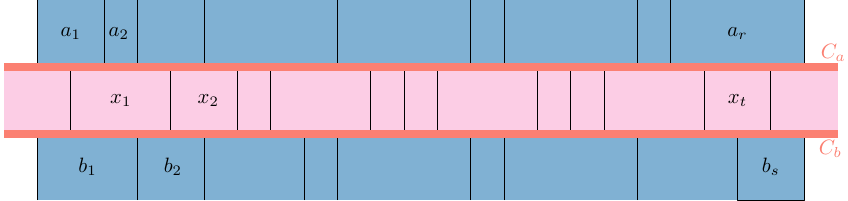} \\[2em]
        \includegraphics{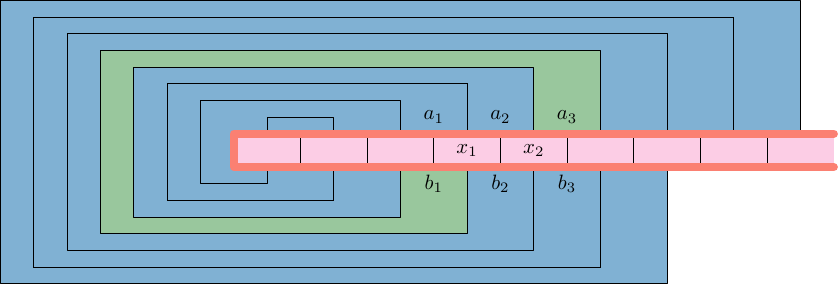}
    }
    \caption{Two illustrations of \obsref{three-paths}}
    \figlabel{three-paths}
\end{figure}

\begin{obs}\obslabel{three-paths}
    Let $x_1,\ldots,x_t$, $t\ge 1$ be a path in $T_0$ consisting entirely of really bad nodes.  Then $C\cap\bigcup_{i=1}^t x_i$ consists of two paths $C_a$ and $C_b$ each having at least one edge and the subgraph of $T_1$ induced by $\bigcup_{i=1}^t N(x_i)$ is contained in two (not necessarily disjoint) paths $a_1,\ldots,a_r$ and $b_1,\ldots,b_s$, $r,s\ge 1$ where $a_i$ contains an edge of $C_a$ for each $i\in\{1,\ldots,r\}$ and $b_i$ contains an edge of $C_b$ for each $i\in\{1,\ldots,s\}$.
\end{obs}

%

\subsection{Tree/Cycle Surgery}

We summarize the situation so far.  By \lemref{cycle_to_curve}, finding a large collinear set is equivalent to finding a cycle in $\dual{G}$ that caresses many faces.  By existing results on the circumference of cubic triconnected graphs, $\dual{G}$ has a cycle $C_0$ of length $\ell=\Omega(n^{\alpha})$ for some $\alpha > 0.8$.  Thus we assume that $\dual{G}$ has a cycle $C_0$ of length $\ell$ and we want to show the existence of a cycle $C$ that caresses $\Omega(\ell/\Delta^4)$ faces.

Because each face of $\dual{G}$ has at most $\Delta$ edges, $C_0$ touches $\Omega(\ell/\Delta)$ faces (\obsref{degree_touched}).  To complete the proof of \thmref{main} we must deal with the situation where $C_0$ caresses $o(\ell/\Delta^4)$ faces and therefore each of $T_0$ and $T_1$ has $o(\ell/\Delta^4)$ leaves (\lemref{one_caressed_leaf}), $\Omega(\ell/\Delta)$ nodes (\lemref{few_caressed_implies_many_nodes}), and the fraction of really really bad nodes in $T_0$ and $T_1$ is $1-o(1/\Delta)$ (\lemref{lots_of_really_really_bad}).

\figref{few_caressed_ii} illustrates an extreme example of this situation.  To handle cases like these, the only option is to perform surgery on the cycle $C$ to increase the number of caressed faces.  We achieve this by performing a surgery that increases the number of leaves in $T_1$. This surgery is quite delicate and requires a particular node $u$ for which we have a good enough understanding of the faces of $\tilde{H}$ surrounding $u$ so that we can make a local modification of $C$ around $N(u)$ that is guaranteed to stricly increase the number of caressed faces.

\begin{figure}
    \centering{
        \includegraphics{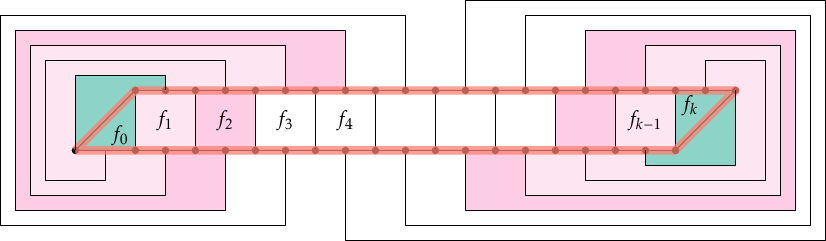}
    }
    \caption{An example in which $C$ caresses only 4 faces of $\dual{G}$, $T_0$ has only 2 non-bad nodes (in teal), 2 non-really bad nodes (in light pink), and 2 non-really really bad nodes (in pink).}
    \figlabel{few_caressed_ii}
\end{figure}

\begin{proof}[Proof of \thmref{main}]
By \lemref{cycle_to_curve}, it suffices to prove the existence of a cycle $C$ in $\dual{G}$ that caresses $\Omega(\ell/\Delta^4)$ faces.  We begin by applying \lemref{lots_of_really_really_bad} with $\alpha = \epsilon/\Delta^3$.  For sufficiently small, but constant, $\epsilon$, \lemref{lots_of_really_really_bad} implies that $\kappa = \Omega(\ell/\Delta^4)$ or the number of nodes in $T_0$ that are not really really bad is at most $O(\epsilon n_0/\Delta)$.  In the former case, $C$ caresses $\Omega(\ell/\Delta^4)$ faces of $\dual{G}$ and we are done.

In the latter case, consider the forest obtained by removing all nodes of $T_0$ that are not really really bad.  This forest has $(1-O(\epsilon/\Delta))n_0$ nodes.  We claim that it also has $O(\epsilon n_0/\Delta)$ components.  To see why this is so, let $L$ be the set of leaves in $T_0$ and let $S$ be the set of non-leaf nodes in $T_0$ that are not really really bad.  Observe that it is sufficient to upper bound the number, $k$, of components in $T_0-S$.

Since removing a degree $d$ vertex from a graph increases the number of components by at most $d-1$, we have $k \le \sum_{u\in S}(\deg_{T_0}(u)-1)$. Since $(S,L)$ is a partition of the nodes of $T_0$ that are not really really bad, we have $|L|\le |S|+|L|=  O(\epsilon n_0/\Delta)$.  Recall that a standard fact about trees is that the number of leaves in a tree $G$ is exactly $2+\sum_{u}(\deg_T(u)-2$, where the sum runs over all non-leaf nodes $u$ of $G$. Therefore,
\[
    |L| \ge \sum_{u\in S}(\deg_{T_0}(u)-2)
     = \sum_{u\in S}(\deg_{T_0}(u)-1) - |S| = k - |S|  \enspace .
\]
Therefore $k\le |S| + |L| = O(\epsilon n_0/\Delta)$, as claimed.

Thus the forest induced by all really really bad nodes of
$T_i$ has at most $O(\epsilon n_0/\Delta)$ components, each of which is
a path.  At least one of these paths contains $\Omega(\Delta/\epsilon)$
nodes. In particular, for a sufficiently small constant $\epsilon$,
one of these components, $X$, has at least $5\Delta$ nodes.


Consider some node $u$ in $X$, and let $C_a$ and $C_b$ be the
two components of $u\cap C$. By \obsref{three-paths}, the subgraph of $T_1$ induced by $N(u)$ consists
of two paths $a_1,\ldots,a_r$ and $b_1,\ldots,b_s$ of really bad nodes
where each $a_1,\ldots,a_r$ contains an edge of $C_a$ and each of
$b_1,\ldots,b_r$ contains an edge of $C_b$.

It follows from \lemref{bad_one_shared_edge} that among any sequence of $\Delta$
consecutive nodes in $X$, at least one node has $r\ge 2$ and therefore $|N(u)|\ge 5$.  Let $u$ be any such node that is not among the first $2\Delta$
or last $2\Delta$ nodes of $X$.  Such a $u$ always exists because $X$
contains at least $5\Delta$ nodes.

Let $x_0=u$. We now define notations for some of the nodes in the vicinity of $u$ (refer to \figref{rrbad}):
\begin{enumerate}
  \item there is a path
    $x_{2\Delta},\ldots,x_1,x_0,y_1,\ldots,y_{2\Delta}$
    in $T_0$ consisting entirely of really really bad nodes.
  \item some really bad node $a_1$ of $T_1$ shares an edge with each of
    $x_0,\ldots,x_i$ for some $i\in\{1,\ldots,\Delta-4\}$.
  \item some really bad node $a_2$ of $T_1$ shares an edge with $a_1$ and
    and edge with $x_0$.
  \item some really bad node $a_0\neq a_2$ of $T_1$ shares an edge with $a_1$ and
    with each of $x_i,\ldots,x_{i+j}$ for some $j\in\{0,\ldots,\Delta-4\}$.
\end{enumerate}

\begin{figure}
   \begin{center}\includegraphics{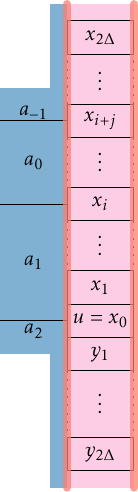}\end{center}
   \caption{Nodes in the vicinity of $u=x_0$.}
   \figlabel{rrbad}
\end{figure}

The surgery we perform focuses on the nodes $u$ and $a_1$.  Consider
the two components of $C\cap a_1$. At least one of these components, $p$, shares
an edge with $u$.  By \lemref{bad_one_shared_edge}, the other component,
$q$, does not share an edge with $u$.  Imagine removing $u$ from $T_0$,
thereby separating $T_0$ into a component $T_x$ containing $x_1$ and a
component $T_y$ containing $y_1$.  Equivalently, one can think of removing
the edges of $u$ from $C$ separating $C$ into two paths $C_x$ and $C_y$
on the boundary of $T_x$ and $T_y$, respectively.  Since $q$ does not share an edge with $u$, $q\subseteq C_x$ or $q\subseteq C_y$.  We treat these cases separately:

\begin{figure}
   \begin{center}
     \begin{tabular}{ccc}
       \includegraphics{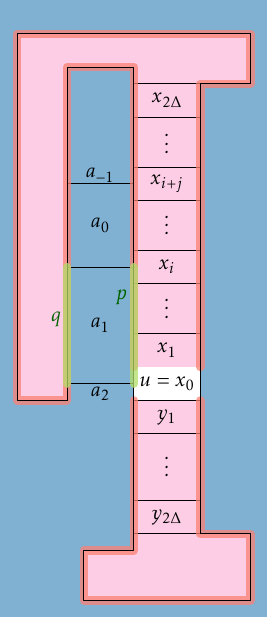} &
       \includegraphics{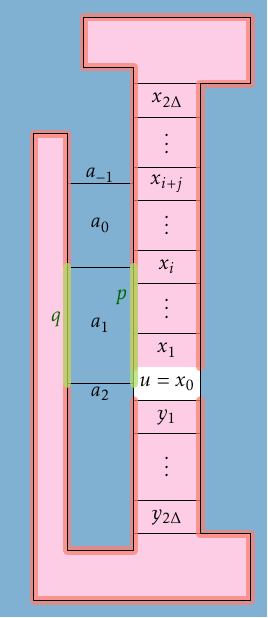} &
       \includegraphics{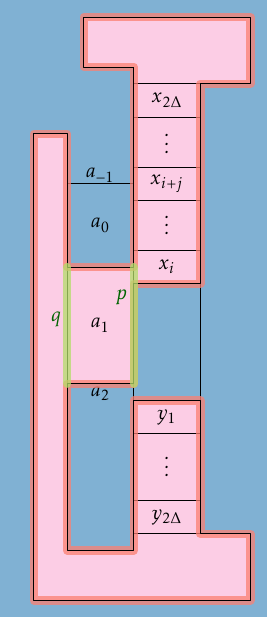} \\
       (1) & (2) & $C'$
     \end{tabular}
   \end{center}
   \caption{Cases 1 and 2 in the proof of \thmref{main} and the surgery performed in Case~2.}
   \figlabel{main_cases}
\end{figure}
\begin{enumerate}
   \item $q\subset C_x$.  We transform this into Case~2, by redefinining $u$, $x_1$, $y_1$ and $a_1$ as follows:  By \lemref{one_shared_edge} $a_1\setminus C$ contains exactly two edges of $\dual{G}$ and exactly one of these edges, $e$, is not incident to $u$. Instead, $e$ is incident to $x_i$.  We set $u'=x_i$, $x_1'=x_{i-1}$, $y_1'=x_{i+1}$, and $a_1'=a_1$. Observe that $a_1'$ connects the two components of $T_0-u'$ and shares edges with $u'$ and $x_1'$. This is exactly the situation considered in Case~2, next.

   \item $q\subset C_y$.  At this point it is helpful to think of $T_0$, $T_1$, and $C$ as a partition of $\R^2$, where nodes of $T_0$ are coloured red, nodes of $T_1$ are coloured blue and $C$ is the (purple) boundary between red and blue.  To describe our modifications of $C$, we imagine changing the colours of nodes.  The effect that such a recolouring has on $C$ is immediately obvious: It produces a 1-dimensional set $C'$ that contains every (purple) edge contained in the red-blue boundary. The set $C'$ is a collection of vertices and edges of $\dual{G}$. Therefore, if $C'$ is a simple cycle, then $C'$ defines a new pair of trees $T_0'$ and $T_1'$.


   Refer to the right two thirds of \figref{main_cases} for a simple
   (and misleading) example of what follows. For a full example,
   refer to \figref{recolouring}.  The surgery we perform recolours
   $x_0,x_1,\ldots,x_{i-1}$ blue and recolours $a_1$ red.  Observe that,
   because $q\subset C_y$ and $p$ contain an edge of $x_i$, this implies
   that the red subset of $\R^2$ is connected. Similarly, one can verify that $T_1-\{a_1\}$ contains two components, one containing $a_2$ and one containing $a_0$ and $b_1$.  The blue subset of $\R^2$ is connected because it contains a path from $a_2$ through $u$ to $b_1$.  Therefore the red and blue subsets of $\R^2$ are each connected and their common boundary $C'$ is a simple cycle consisting of edges of $\dual{G}$.  The new trees $T_0'$
   and $T_1'$ are therefore well defined.  We now make two claims that
   will complete our proof.

   \begin{figure}
     \begin{center}
       \begin{tabular}{cc}
         \includegraphics{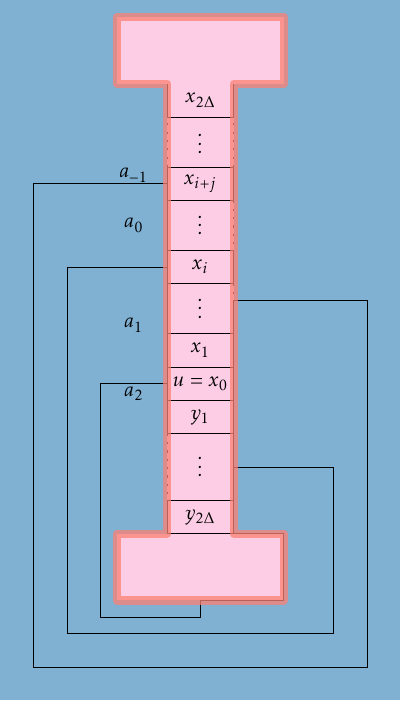} &
         \includegraphics{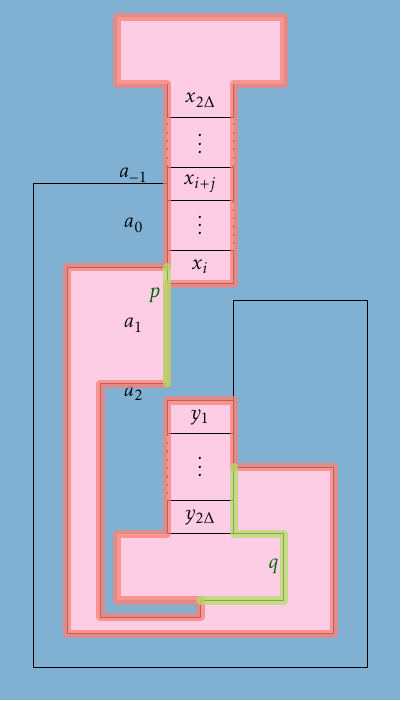}
       \end{tabular}
     \end{center}
     \caption{Performing surgery on $C$ to obtain $C'$ that caresses $a_0$.}
     \figlabel{recolouring}
   \end{figure}

   \begin{clm}\clmlabel{only_bad}
      For each $i\in\{0,1\}$, and each node $w$ of $T_i$ that is not bad,
      $C\cap w=C'\cap w$.  (Equivalently, for every face $f$ of $\dual{G}$
      that is not a bad node of $T_0$ or $T_1$, $C\cap f=C'\cap f$.)
   \end{clm}

   \begin{clm}\clmlabel{a0_caressed}
      The face $a_0$ is caressed by $C'$.
   \end{clm}

   These two claims complete the proof because, together, they imply that $C'$ caresses at least one more face of $\dual{G}$ than $C$.  Indeed, by definition, $C$ did not caress any faces belonging to bad nodes. Therefore, the first claim implies that the faces of $\dual{G}$ caressed by $C'$ are a superset of those caressed by $C$.  The face $a_1$ is a bad node of $T_i$ so it is not caressed by $C$ but the second claim states that it is caressed by $C'$.  Therefore $C'$ caresses at least one more face than $C$.

   This surgery recolours at most $\Delta-2\le \Delta$ nodes of $T_0$
   and $T_1$, so the difference in length between $C$ and $C'$
   is at most $\Delta^2$.  If we start with a cycle $C$
   of length $\ell$, then we can perform this surgery at least
   $\ell/(4\Delta^2)$ times before the length of $C$ decreases
   to less than $\ell'=\ell/2$. If at some point during this process,
   we are no longer able to perform this operation, it is because
   $C$ caresses $\Omega(\ell'/\Delta^4)=\Omega(\ell/\Delta^4)$
   faces of $\dual{G}$ and we are done.  If the process runs to completion, then by its end, the number of faces caressed by $C$ is at least
   $\ell/(4\Delta^2)\in\Omega(\ell/\Delta^2)\subset\Omega(\ell/\Delta^4)$
   and we are also done.

   Thus, all that remains is to prove \clmref{only_bad} and
   \clmref{a0_caressed}.

   To prove \clmref{only_bad}, we observe that $C$ and $C'$ differ only on
   the boundaries of nodes that are recoloured.  Thus, it is sufficient
   to show that all nodes in $R=\cup\{N(v): v\in\{x_0,\ldots,x_{i-1},a_1\}$
   are bad.  But this is immediate since $x_0,\ldots,x_{i-1}$ are really
   really bad and $a_1\in N(x_0)$, so $a_1$ is really bad.  Since every node in $R$ share an edge with at least one of $\{x_0,\ldots,x_{i-1},a_1\}$, every
   node in $R$ is therefore bad, as required.

   To prove \clmref{a0_caressed} we consider the boundary of the face
   $a_0$ of $\dual{G}$ after the recolouring operation.  This boundary
   consists of, in cyclic order:
   \begin{enumerate}
     \item  An edge $p_0p_1$ shared between $a_0$ and $a_1$.  This edge is in
       $C'$ since $a_0$ is in $T_1'$ and $a_1$ is in $T_0'$. This edge has one endpoint, $p_0$, on the boundary of $x_i$ ($p_0$ is also an endpoint of $p$).

     \item A path $p_1,\ldots,p_\mu$ whose edges are shared with $x_i,\ldots,x_{i+j}$.  The nodes $x_i,\ldots,x_{i+j}$ are in $T_0$ and are distinct from $x_0,\ldots,x_{i-1}$, so these nodes are in $T_0'$.
     Therefore, $p_1,\ldots,p_\mu$ is also contained in $C'$.

     \item An edge $p_\mu p_{\mu+1}$ shared between $a_0$ and another node $a_{-1}\neq
     a_1$ of $T_1$. The faces of $a_{-1}$ are in $T_1'$ because $a_1$
     is the only face that moves from $T_1$ to $T_0'$. ($a_1$ is the
     only face whose colour goes from blue to red.)  The edge $p_\mu p_{\mu+1}$ is therefore not contained in $C'$.

     \item A path $p_{\mu+1},\ldots,p_\nu$ with $p_\nu=p_0$ that is contained in $C$. Let $C_x'$ be the path obtained by removing all edges on the boundary of $x_1,\ldots,x_{i-1}$ from $C_x$.  Thus, the boundary of $C$ is partitioned into four paths: $C_y$; a path $P_1$ that contains $p$; $C_x'$; and a path $P_2$ that does not contain $p$.  Without loss of generality, assume that these four paths occur in the order $C_y,P_1,C_x',P_2$ when traversing $C$ clockwise.

     The path $p_{\mu+1},\ldots,p_{\nu}$ ends at $p_\nu=p_0$, which is contained in $C_y$. This path must therefore either begin in $P_2$ or be entirely contained in $C_y$ since, otherwise it would contain an edge of $x_i$, contradicting \lemref{one_shared_edge}.  The edges of $P_2$ are not in $C'$. Therefore $p_{\mu+1},\ldots,p_{\nu}$ begins with a (possibly empty) sequence of edges $p_{\mu+1},\ldots,p_{\mu+k}$ not contained in $C'$ followed by a non-empty sequence $p_{\mu+k},\ldots,p_\nu$ of edges that are contained in $C'$.
  \end{enumerate}
  Therefore the intersection $C'\cap a_0$ is a path $p_{\mu+k},\ldots,p_\nu,p_1,\ldots,p_{\mu}$ so $a_0$ is caressed by $C'$. \qedhere
\end{enumerate}
\end{proof}


\section{Discussion}
\seclabel{discussion}


It remains an open problem to eliminate the dependence of our results
on the maximum degree, $\Delta$, of $G$.  The next significant step
is to resolve the following conjecture:

\begin{conj}
  If $G$ is a triangulation whose dual $\dual{G}$ has a cycle of length
  $\ell$, then $\dual{G}$ has a cycle that caresses $\Omega(\ell)$
  faces. (Therefore, by \lemref{cycle_to_curve} and \thmref{dalozzo},
  $G$ has a collinear set of size $\Omega(\ell)$.)
\end{conj}

\section*{Acknowledgement}

Much of this research took place during the Sixth Workshop on Order and
Geometry held in Ciążeń, Poland, September 19--22, 2018.  The authors
are grateful to the organizers, Stefan~Felsner and Piotr~Micek, and to
the other participants for providing a stimulating research environment.

\bibliographystyle{plain}
\bibliography{collsets}

\end{document}